\documentclass[12pt]{amsart}
\usepackage{latexsym,fancyhdr,amssymb,color,amsmath,amsthm,graphicx,listings,comment}
\usepackage[section]{placeins}
\newtheorem{thm}{Theorem} \newtheorem{lemma}{Lemma} \newtheorem{propo}{Proposition}
\let\paragraph\subsection

\title{The Tree-Forest Ratio}
\author{Oliver Knill}
\date{May 22, 2022}
\address{Department of Mathematics \\ Harvard University \\ Cambridge, MA, 02138 }


\begin{document}
\maketitle

\begin{abstract}
The number of rooted spanning forests divided by the number of spanning rooted trees in a graph $G$ with
Kirchhoff matrix $K$ is the spectral quantity $\tau(G)={\rm det}(1+K)/{\rm det}(K)$ of $G$ 
by the matrix tree and matrix forest theorems. 
We prove that that under Barycentric refinements, the tree index $T(G)=\log({\rm det}(K))/|G|$ and
forest index $F(G)=\log({\rm det}(1+K))/|G|$ and so the tree-forest index $i=F-G=\log(\tau(G))/|G|$ converge
to numbers that only depend on the size of the maximal clique in the graph. 
In the one-dimensional case, all numbers are known: 
$T(G)=0, F(G)=i(G) =2 \log(\phi)$, where $\phi$ is the golden ratio. 
The convergent proof uses the Barycentral limit theorem assuring the Kirchhoff spectrum 
converges weakly to a measure $\mu$ on $[0,\infty)$ that only depends on dimension of $G$. 
Trees and forests indices are potential values $i = U(-1)-U(0)$ for the subharmonic function
$U(z)=\int_0^{\infty} \log|x-z| \; d\mu(x)$ 
defined by the Riesz measure $d\mu=\Delta U$ which only depends on the dimension of $G$.
The potential $U(z)$ is defined for all $z$ away from the support of $d\mu$ and finite 
at $z=0$. Convergence follows from the tail estimate $\mu[x,\infty] \leq C e^{-a_d x}$,
where the decay rate $a_d$ only depends on the maximal dimension. With the normalized zeta function 
$\zeta(s) = \frac{1}{|V(G)|} \sum_k \lambda_k^{-s}$, we have for all finite graphs of maximal dimension $\geq 2$
the identity $i(G) = \sum_{t=1}^{\infty} (-1)^{s+1} \zeta(s)/s$.
The limiting zeta function $\zeta(s) = \int_{0}^{\infty} x^{-s} d\mu(x)$ is analytic in $s$ for $s<0$. 
The Hurwitz spectral zeta function $\zeta_z(s)=U_s(z) = \int_0^{\infty} (x-z)^{-s} \; d\mu(x)$ complements
$U(z) = \int_0^{\infty} \log(x-z) \; d\mu(x)$ and is analytic for $z$ 
in $\mathbb{\mathbb{C}} \setminus \mathbb{R}^+$ and for fixed $z$ in $\mathbb{C} \setminus \mathbb{R}^+$
is an entire function in $s \in \mathbb{CC}$. 
\end{abstract}

\section{In a nutshell}

\paragraph{}
We define here the {\bf tree forest ratio} $\tau(G)$ of a connected finite simple graph $G=(V,E)$
as the number of rooted spanning forests in $G$ divided by the number of rooted spanning trees 
in $G$. By the tree and forest matrix theorems, this is
$\tau = {\rm Det}(K+1)/{\rm Det}(K)$, where ${\rm Det}$ is the pseudo determinant and 
$K$ the Kirchhoff Laplacian of $G$. In other words, it is
$\tau=\prod_{\lambda \neq 0} (1+1/\lambda)$, where $\lambda$
runs over the non-zero eigenvalues of $K$. 
An upper bound is $(1+\lambda_2)^{-1})^|V|$, where $\lambda_2$
is the {\bf ground state}, the smallest non-zero eigenvalue of $K$. 
An expansion of the logarithm and a basic bound on the product gives
$\log(1+\zeta(1)) \leq \log(\tau) = \sum_{s=1}^{\infty} (-1)^{s+1} \zeta(s-1)/s \leq \zeta(1)$, 
where $\zeta(s)$ is the Kirchhoff spectral zeta function of $G$. 

\paragraph{}
The same functional could be considered for other operators like the 
connection matrix $L$ or the Hodge matrix $H$ of a graph, even so a tree or forest interpretation lacks then.
For the connection matrix which has negative eigenvalues in general, we would use $L^2$ to have no ambiguity 
when defining zeta functions
As we have an equivalent zeta function expression in the Kirchhoff case, the function
$\tau(G)$ can now make sense even for manifolds $G$ which do not feature spanning trees or spanning forests.
The reason is that the Laplacian on the manifold defines a Minakshisundaram - Pleijel zeta function and 
a Ray-Singer determinant. For $M=\mathbb{T}=\mathbb{R}/\mathbb{Z}$, where $\zeta$ is the 
{\bf classical Riemann zeta function}, we have $\tau_G = \sinh(\pi)/\pi$. 

\paragraph{}
Potential theory comes in when seeing the 
logarithm of the {\bf tree number} as a potential value $V(0)=\log({\rm Det}(L))$. Also the 
logarithm of the {\bf forest number} $V(-1)=\log({\rm Det(1+L)})$ is a {\bf potential value} of the potential 
$V(z)=\log({\rm Det}(L-z I)) = \int \log(x-z) \; dk(x)$, where $dk=\sum_j \delta_{\lambda_j}$ 
is a finite pure point measure defined by the eigenvalues of $L$. When normalizing $dk$ to become a
probability measure $d\mu$, then $U(z) = \int \log(x-z) \; d\mu(x)$ has a subharmonic real part
and $d\mu$ is the {\bf Riesz measure} $d\mu= \Delta U$, a probability measure with support on 
the positive real axes. 
The Barycentral limit theorem assures that $d\mu_n$ converges in the Barycentric limit weakly
to a measure $d\mu$ which only depends on the maximal dimension of the initial graph. 
The potential value $U(z)$ is defined for all $z$ away from the support of $d\mu$ and finite  
at $z=0$ because it is there the limiting growth rate of trees which is boxed in by $0$ and 
the growth rate $U(-1)$ of forests.  In the $1$-dimensional case, it is the arc-sin 
distribution on $[0,4]$ which is the potential theoretical equilibrium measure on that interval. 

\begin{figure}[!htpb]
\scalebox{0.2}{\includegraphics{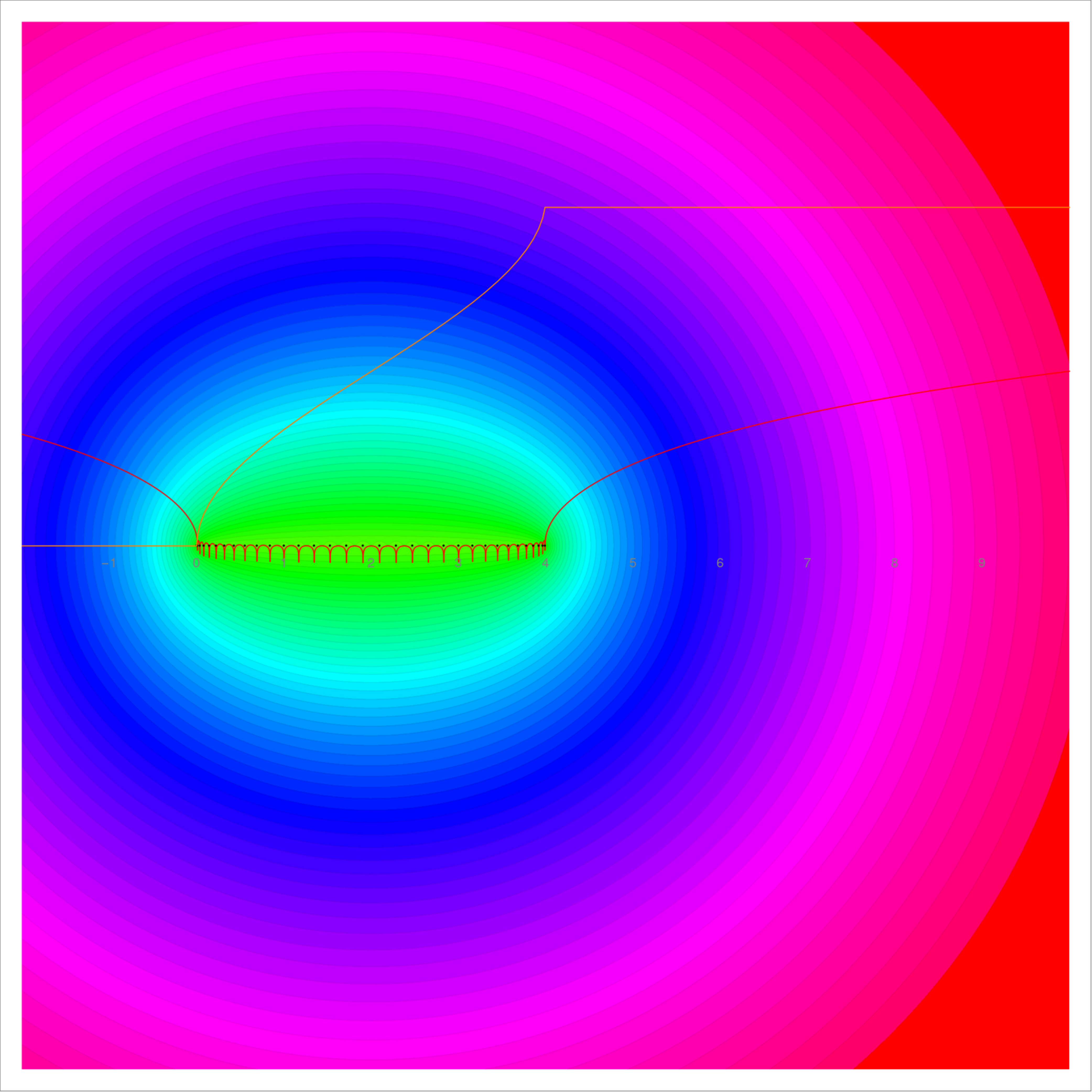}}
\scalebox{0.2}{\includegraphics{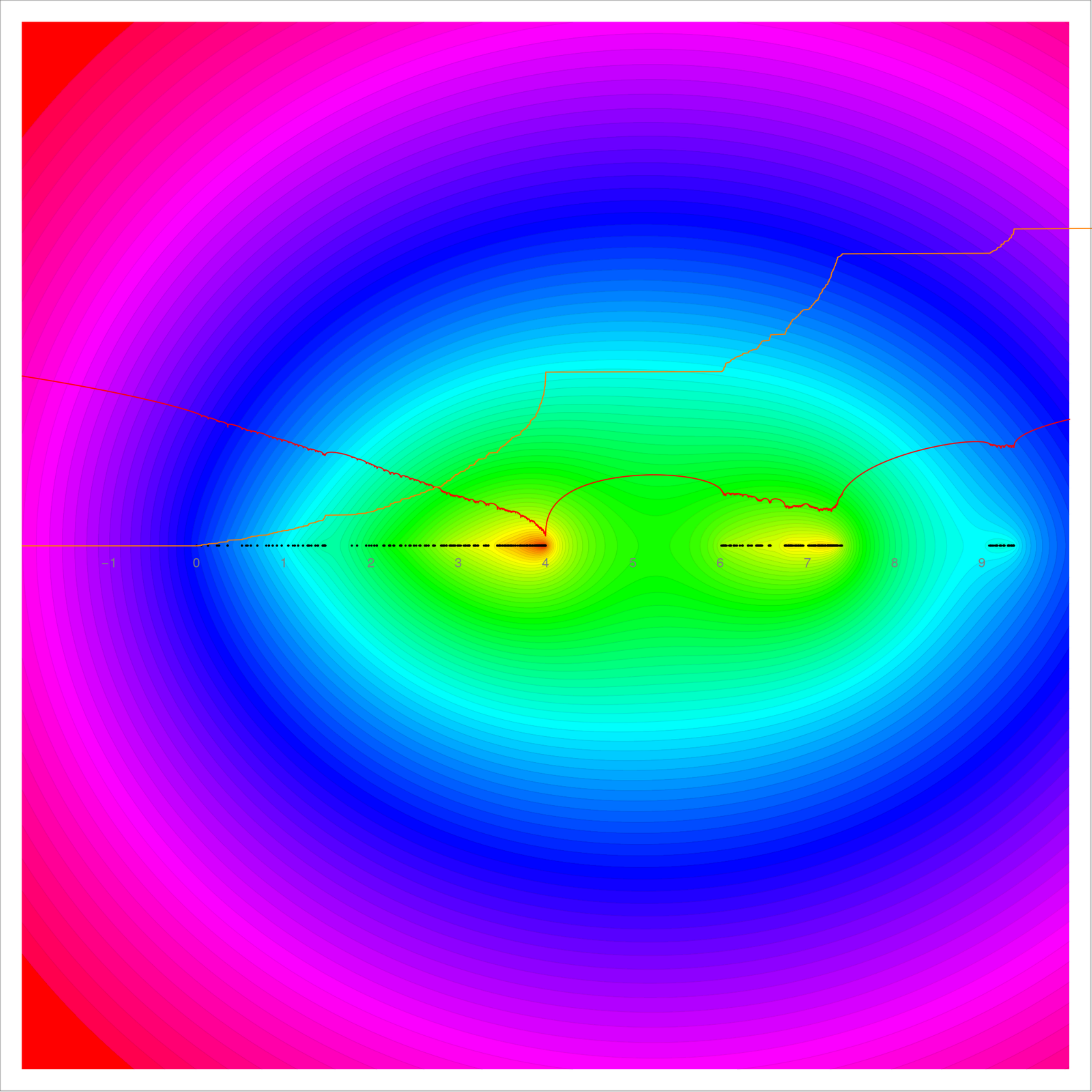}}
\label{potential}
\caption{
The real and imaginary part of the potential in the 1-dimensional case is seen to the left.
The support of the density of states is $[0,4]$. In the 2-dimensional case, where
the support of the density of states is non-compact, we see gaps and the support of the spectrum
could be a Cantor set. 
}
\end{figure}

\paragraph{}
A major point we want to make here is to see that the limiting universal measure $d\mu$ has 
an exponentially decaying {\bf tail distribution} $\mu[x,\infty) \to 0$ 
as $x \to \infty$. This implies that both the forest and tree numbers
converge in the Barycentral limit because the potential values $U(-1)$ and $U(0)$ exist. Since
by nature of tree and forest numbers $0 \leq U(0) \leq U(-1)$, the convergence at the bottom of the
spectrum $z=0$ is not a problem. The tail estimate also shows that for $z$ away from the 
positive real axes, the {\bf Hurwitz zeta function} 
$\zeta_z(s) = U_s(z) = \int_0^{infty} (x-z)^{-s} \; d\mu(x)$ is an entire
function in $s$. For $z=0$, we only know that $\zeta_0(s)=\zeta(s)$ is analytic for $s<0$.

\paragraph{}
We will use some linear algebra to estimate the integrated density of states near 
$\infty$ in terms of the vertex degree distribution of the graph. 
We need a bit more than the {\bf Schur-Horn inequality} or the 
{\bf Gershgorin circle theorem}: we use $\lambda_k \leq 2 d_k$, where $d_k$ is 
the $k$'th vertex degree. This appears to be a new 
result we have written down separately and which follows directly from the Cauchy 
interlace theorem. We have looked at the zeta function $\zeta(s)$ 
of circular graphs $G=C_n$ in \cite{KnillZeta} and looked at the limiting behavior of roots and 
then in \cite{DyadicRiemann} for the connection graphs of circular graphs, where the connection 
Laplacian is invertible \cite{Unimodularity,KnillEnergy2020} and the Zeta function satisfies 
a functional equation. 

\paragraph{}
For small graphs like complete graphs $K_n$ or the graph complements 
$C_n'$ of cyclic graphs $C_n$, the tree-forest ratio converges 
to the {\bf Euler constant} $e=2.718 \dots$ in the limit $n \to \infty$. 
For $K_n$, one can see that directly from the eigenvalues $\lambda_0=0,\lambda_k=n$ because
$\tau(G)=(1+1/n)^{n-1} \to e$. For Barycentric refinements of a graph $G=G_0$ 
the tree forest index $i(G_n)=\log(\tau(G_n))/|V(G_n)|$ of $G_n$ converges to a universal 
constant that only depends on $d$. Since $\sum_{k=0}^d f_k(G_{n-1})=|G_n|=f_d \leq \sum_{k=0}^d f_k(G_n)$
we could also take the limit, where $|V(G|$ is replaced by any $f_k$ like for example 
the volume $f_d$. The result is a direct consequence of {\bf spectral Barycentral limit theorem}
\cite{KnillBarycentric,KnillBarycentric2}. In the Barycentric limit, we have 
$\log(\tau(G_n))/|G_n| \to  \int_0^\infty \log(1+1/x) d\mu(x)$, where $\mu$ is the limiting
density of states. To show that this is a finite number, we only need to show that 
the density of states $dk$ decays fast. 
We needed something stronger than the Gershgorin circle theorem \cite{Gershgorin,GershgorinAndHisCircles}
or the Schur inequality in matrix theory \cite{Brouwer} and got $\lambda_k \leq d_k$. 
For $d=1$, we can compute this limiting constant as
$\lim_{n \to \infty} i(G_n) = 2 \log(\phi)$, where $\phi=(1+\sqrt{5})/2$ is the golden ratio.
When looking at Schur for the sum of the first $n$ eigenvalues and the 
Cheeger inequality for the ground state $\lambda_1$,
it is natural to conjecture that the {\bf linear graph} with $n$ vertices and $(n-1)$ edges
strictly maximizes $\tau(G)$ among all connected graphs with $n$ vertices. 

\paragraph{}
Having made more experiments with the limiting density of states in the case $d=2$ in particular,
it is numerically indicative that $[4,6]$ is the {\bf largest gap} in the Barycentric spectral 
measure for $d=2$. If we take a triangle $G=K_3$ and make Barycentric refinements, then after $2$ or more
Barycentric refinements, exactly half of the positive eigenvalues are in the interval $[0,4]$ and half of the
positive eigenvalues are in $[6,\infty)$. It is natural to ask whether there is a {\bf gap labeling}: the {\bf integrated
density of states} $\int_0^x d\mu(x)$ is a rational number for $x$ in a gap of the spectrum. But already in the case
$d=2$ for the relatively large gap containing the point $8$, we don't see an obvious small rational number attached. Our best guess 
for the integrated density of states is $\int_0^8 dk(x) = 137/164$. 
Gap labeling results are prototyped by p-periodic Jacobi matrices, where the integrated density of states is $j/p$ with integer
$j$. For classes of non-periodic operators like almost periodic operators, the gap labeling works too. 
For early sources,see \cite{Pastur,Cycon,Carmona,Bel+92a}. For a recent development interesting for graph theory is the case 
of periodic Jacobi matrices on universal covers of leaf-free one-dimensional graphs \cite{AvniBreuerSimon}, 
where one observes only point or absolutely continuous spectrum. In that respect, we have absolutely no idea yet what
the spectral type of the Barycentric density of states $dk$ is in dimension $d \geq 2$. Observing that the 
potential $U(z)$ (the analogue of the Lyapunov exponent in the case of Schr\"odinger operators) appears to be positive 
almost everywhere on the support of $dk$ if the dimension is larger than $1$. 

\paragraph{}
In the case of Jacobi matrices, one is interested in the spectral measures of random operators $L(\omega)$ parametrized by 
a point in a probability space $\Omega$ which in the almost periodic case is a compact topological group with Haar measure.
We believe that there is a non-Abelian compact topological group in all dimensions. It is hidden for $d \geq 2$ still.
There should be Laplacian on this group such that $dk_d$ is the
density of states of a ``random" Laplacian. In the case $d=1$, where things are Abelian and understood,
the group $\mathcal{G}_1$ is the {\bf dyadic group of integers} and the limiting
operator is the extension of $H f(n) = 2-f(n+1)-f(n-1)$ from the dense set $\mathbb{Z}$ to its completion $\mathcal{G}_1$. 
By Fourier theory, $H$ is diagonalized and unitarily equivalent to 
$\hat{H} f(x) = 2-2\cos(\pi x) = 4\sin^2(\pi x/2)$ on the {\bf Pr\"ufer group} $\hat{\mathcal{G}}_1$ of all dyadic rationals on $\mathbb{T}$.
While $dk=1/(\pi \sqrt{4(4-x)})$ on $[0,4]$ is clearly absolutely continuous, the operator $H$ has dense pure point spectrum as
an operator on $L^2(\mathcal{G}_1)$. Unlike in the Pontryagin pair $\mathbb{Z} \leftrightarrow \mathbb{T}$ where the Dirac points
$\delta_x$ in $\mathbb{T}$ are distributions and not functions in in $L^2(\mathbb{T})$, the Dirac points 
$\delta_x in l^2(\hat{\mathcal{G}}_1)$ are actual functions as the Pr\"ufer group is discrete. In the Jacobi case, the operator
was diagonal on the compact, non-discrete group of the Pontryagin pair, in the dyadic case, the operator is diagonal on 
the discrete-non-compact case. 

\section{Introduction}

\paragraph{}
Investigating the relation between spectral data of a geometry and observable quantities
like topological invariants is a classical theme in geometry \cite{BergerPanorama}.
It was popularized by Mark Kac whether one can ``hear the geometry" 
\cite{Kac66}. Historically, a prototype result is {\bf Weyl's law} \cite{Weyl1911} relating the number of Dirichlet 
eigenvalues $n$ of the Dirichlet Laplacian on $G$ below $\lambda$ with the volume $|G|$ of a domain $G$ as 
$n \sim C_d \lambda^{\frac{d}{2}} |G|$ with $C_d = \frac{|B_d|}{2\pi^d}$.
Herman Weyl looked in \cite{Weyl1911} at the planar case and proved $n/\lambda_n = |G|/(4\pi)$ 
using the explicit eigenvalues $\lambda_{m,n}=\pi^2(m^2+n^2)/|G|$ for a square $G$. Kac's question addresses
the problem to what degree the spectral data determine space and how large the set of objects
is that are isospectral. Because the zeta function encodes the spectrum, this also means
to what degree the spectral zeta function determines the geometry.

\paragraph{}
Spectral questions are also related to inverse problems like to find the isospectral set in classes
of geometries.  While in one-dimensional cases, one can deform geometry in an isospectral way, 
in higher dimensions, isospectral rigidity appears. Isospectral sets tend to be discrete
as isospectral deformations of Riemannian geometries in general do not exist any more
\cite{Mor1}. For finite simple graphs, the question of how spectral data relate to
a given invariant is largely parallel to the continuum theory \cite{Chung97} and 
can be investigated experimentally. A typical spectral result is that the pseudo
determinant ${\rm Det}(K)$ of the Kirchhoff Laplacian of a graph is the number of rooted 
spanning trees and the determinant of the Fredholm modification is the number of rooted 
spanning forests. We look here at this ratio. In this context, we can also remind 
about the spectral properties of connection matrices, where the 
number of positive eigenvalues minus the number of negative eigenvalues is the Euler characteristic
\cite{HearingEulerCharacteristic}.

\paragraph{}
The {\bf tree forest ratio} is picked up here also because it belongs to a larger quest 
to investigate {\bf functionals on graphs}, that is assigning numerical values to a graph.
Considering critical points of functionals on geometries has 
been successful in physics, like for geodesics, area minimizing surfaces, total curvature minimizing
In graph theory, one can look at ``packing numbers" like the chromatic number or the independence number
manifolds. In graph theory, quantities that can be
counted are the average simplex cardinality \cite{AverageSimplexCardinality}, 
the number $f_k$ of $k$-dimensional simplices or the 
Euler characteristic$ \chi(G)=\sum_k (-1)^k f_k$ or other length-related functionals
\cite{KnillFunctional}. Interesting examples of {\bf spectral quantities} is the ground state energy, 
the smallest positive eigenvalue of the Laplacian or then the {\bf analytic torsion} 
which involves spectral data of the Hodge Laplacian $L=(d+d^*)^2$. 

\section{Trees and Forests}

\paragraph{}
Given a {\bf finite simple graph} $G$, the collection of {\bf rooted spanning trees} and 
the set of {\bf rooted spanning forests} inside $G$ are of interest. For forests, 
``rooted" means that every tree in the forest has an assigned vertex, which is interpreted 
as the {\bf root} of the tree. 
Rooted trees in a forest include also seeds, which are one-point graphs. Each rooted tree can be seen
as pointed topological space without closed loops nor triangles. It is a directory organization of 
the vertices and a stratification of space by giving the distance from the root. 
For any spanning tree $T$ or spanning forest, the induced graph is the full space $G$
because $T$ spans $G$: the vertex set of $T$ and $G$ is the same. The ratio of
rooted spanning forests and trees is the same than the ratio of spanning forests and trees.  

\paragraph{}
By the {\bf Kirchhoff matrix tree theorem}, the
number of rooted spanning trees in $G$ is ${\rm Det}(K)$, where $K$ is 
the {\bf Kirchhoff matrix} of $G$ and ${\rm Det}$ is the {\bf pseudo determinant} of $K$,
the product of the non-zero eigenvalues of $K$. Since every spanning tree has
$n$ possible roots if the graph has $n$ vertices, the number ${\rm Det}(K)/n$
is the number of spanning trees in $G$. 

\paragraph{}
By the {\bf Chebotarev-Schamis forest theorem} 
\cite{ChebotarevShamis1,ChebotarevShamis2,Knillforest}, the number of rooted spanning
forests in $G$ is ${\rm det}(1+K)$, the {\bf Fredholm determinant} of the 
Kirchhoff matrix $K$. See \cite{cauchybinet}. 
We introduce here the {\bf tree-forest ratio}
$$    \tau(G) = \frac{{\rm det}(1+K)}{{\rm Det}(K)} $$
and the {\bf tree forest index}
$$    i(G) = \log(\tau(G))/|G|  \;  $$
where $|G|$ is the number of vertices. 
In general $\tau(G) \geq 0$ because there are more forests than trees.

\paragraph{}
In dimension larger than $1$, both the number of forests as well as the number of trees grow 
exponentially when doing Barycentric refinements.
The reason is that in the two dimensional case already, we can see this geometrically.
This is different for one-dimensional graphs, where the number of trees grows polynomially
under Barycentric subdivisions while the number of forests 
grows exponentially. The precise polynomial 
degree of the tree growth rate in the one-dimensional case depends
on the genus $b_1$ of the graph. 

\paragraph{}
For graphs with multiple connected components, where spanning trees naturally
need to be disconnected, one can define the number of rooted spanning trees as the product 
$\prod_{k} {\rm Det}(K_k)$, where $K_k$ is the Kirchhoff
matrix block belonging to the $k$'the component. For the number of rooted forests, we
naturally have the product $\prod_k {\rm Det}(K_k+1)$ already because forests do not need
to be connected. With this extension of the definition for not necessarily connected graphs $G$,
the ratio $\tau(G)= {\rm Det}(K+1)/{\rm Det}(K)$ can in general be understood as the tree-forest ratio, 
whether $G$ is a connected or a disconnected graph. To make the tree-forest ratio 
functional defined on all graphs, we define $\tau(0)=1$ and $i(0)=0$ if $0$ is the empty graph. 

\paragraph{}
Since all rooted spanning trees are also rooted spanning forests, 
we have $\tau(G) \geq 1$ and so $i(G)=\log(\tau(G))/|G| \geq 0$. It is easy to check 
that $\tau(G)=1$ happens if and only if $G$ has no edges. The reason is 
that already the presence of one single edge produces more forests than trees.
For the complete graph $K_2$ for example, we have
$\tau(K_2) = 3/2$ because there are $3$ rooted forests and $2$ rooted trees 
in $K_2$. 

\paragraph{}
When taking graph limits with larger and larger trees, 
the tree forest ratio diverges in general exponentially fast with 
the diameter. For the cycle graph $C_n$ for example, there are $n^2$ rooted spanning trees
and $n$ spanning trees. In the same circular graph $C_n$,
the number of spanning forests is the alternate {\bf Lucas number} $L(n)$ 
given by the recursion $L(n+1) = 3 L(n)-L(n-1)+2, L(0)=0,L(1)=1$. 
The number of forests grows exponentially, while the number of trees
grows polynomially. We have 
$$   \tau(C_n)=n^2/L(n) \;  $$
and 
$$  \log(\tau(C_n))/n \to 2 \log(\phi) = {\rm arccosh}(3/2) \; , $$
where $\phi$ is the {\bf golden ratio}.

\paragraph{}
For any one-dimensional graph $G$ with $|G|=|E|$ edges, the 
{\bf Barycentral limit measure} $d\mu$ 
exists explicitly. The limiting density of states measure is the arcsin distribution.
Potential theoretically, the fact that that the number of trees grows polynomially while
the number of forests grows exponentially is reflect in 
$\int_0^1 \log(4 \sin^2(\pi x)) \; dx = 0$ which is not totally obvious. 
In terms of the density of states $dk(x) = \frac{1}{\pi \sqrt{x(4-x)}}$ we have
$$ U(0) = \int_0^4 \frac{\log(x)}{\pi \sqrt{x(4-x)}} \; dx = 0 \;  $$ 
and 
$$   U(-1) = \int_0^4 \frac{\log(1+x)}{\pi \sqrt{x(4-x)}} \; dx 
= \log(\phi^2) = {\rm arccosh}(3/2) =0.962424 \; . $$


\paragraph{}
Here is an other class, of graphs where we have a chance to compute the limit.
One can look at {\bf circulant graphs} $C_{n,A}$, 
the Cayley graph of a generating set $A \subset V(C_n)$. If the diameter
of the graph is larger than 2, then the tree-forest ratio is going
to explode for such graphs The radius of a circulant graph is $2$ if the set $A+A$ 
in $\mathbb{Z}_p$ covers the entire set $\mathbb{Z}_p$. An extreme
case is if $A=V(C_n)$, in which case we have the complete graph. An other 
example is the graph complement $C_n'$ of $C_n$ in which case the diameter
is always $2$. 
If the set $A+A$ for the generating set of the Cayley graph does not cover the graph,
meaning that there are relations in the finitely generated group with words $3$
or longer, then the tree-forest ratio of $G_n$ diverges exponentially
as $n \to \infty$. In some cases, we can compute the limit. 
For the graph complement $\overline{C}_n$ of cyclic graphs $C_n$, which are graphs
of diameter $2$, one has:

\begin{propo}
$\tau(\overline{C}_n) \to e$.
\end{propo}
\begin{proof}
The eigenvalues of $\overline{C}_n$ are explicitly known as
$$  \lambda_{k,n} = \sum_{m=2}^{n-2} 1-\cos(2\pi m \frac{k}{n}) 
                  = \sum_{m=2}^{n-2} 2\sin^2(\pi m \frac{k}{n}) \; . $$
We have $2 \lambda = 5-2+2\cos(2\pi k/n) - \frac{\sin(k (2m-3) \pi/n }{\sin(k \pi/n)}$.
\end{proof}

\begin{propo}
$\tau(K_n) \to e$. 
\end{propo}
\begin{proof}
The non-zero eigenvalues are all $n$. We have
$\tau(K_n) = {\rm Det}(1+K)/{\rm Det}(K) = (1+1/n)^{n-1}$. 
\end{proof}

\section{The Kirchhoff spectral Zeta function}

\paragraph{}
The {\bf Kirchhoff spectral zeta function} of a finite simple graph $G=(V,E)$ is defined as 
$$ \zeta(s) = \sum_{\lambda \neq 0} \frac{1}{\lambda^s} \; ,$$ 
where $\lambda$ ranges over the eigenvalues of the {\bf Kirchhoff Laplacian} $K$
of $G$. We see immediately from the definitions that 
$$  \tau(G) = \prod_{\lambda \neq 0} (1+\frac{1}{\lambda}) \; . $$
When taking the Barycentral limit, we need to normalize the function, usually by scaling with 
a factor $1/n$, where $n$ are the number of vertices in the graph. 

\paragraph{}
It follows from the definition that we can estimate the {\bf spectral gap}
$\lambda_1$, which is the smallest eigenvalue of $K$ from below.

\begin{propo}
$\lambda_1 \geq 1/(\tau(G)-1)$. 
\end{propo}

\begin{proof}
We have             
$(1+1/\lambda_1) = \tau(G)/\prod_{k \geq 2} (1+1/\lambda_k) \leq \tau(G)$.
So $\lambda_1 \geq 1/(\tau(G)-1)$.  
\end{proof}

\paragraph{}
Since $h(G) \geq \lambda_2/2$, where $h(G)$ is the {\bf Cheeger constant}
we also have 

\begin{propo}
$h(G) \geq 1/2(\tau(G)-1$.
\end{propo}

\paragraph{}
Also in reverse, a very small spectral gap $\lambda_1$ produces a large $\tau$ 
because $\tau(G) \geq 1+1/\lambda_1$. 

\paragraph{}
In gteneral, we have for a finite simple graph with (not yet normalized) zeta
function: 

\begin{propo}
$\log(\tau(G)) = \sum_{k=1}^{\infty} (-1)^{k+1} \frac{\zeta(k)}{k}$. 
\end{propo}
\begin{proof}
Taking logs gives 
$$ \log(\tau(G)) = \sum_{\lambda \neq 0} \log(1+ \frac{1}{\lambda}) \; . $$
Using the expansion $\log(1+x) = x-x^2/2+x^3/3+ \dots$ and Fubini, we see
$$ \log(\tau(G)) = \zeta(1)-\zeta(2)/2-\zeta(3)/3 + \dots \; . $$
\end{proof}

\paragraph{}
The Kirchhoff Laplacian $K$ is one possible choice which comes natural when looking at 
trees or forests. We could also look at the {\bf Hodge Laplacian} $L=D^2$ with 
{\bf Dirac operator} $D=d+d^*$ of a graph $G$. The later has 
symmetric spectrum $-\lambda_n,  \dots, -\lambda_1, 0, \lambda_1 , \lambda_n$. 
This gives a {\bf Hodge zeta function} $\zeta_L(s)$. Define the {\bf Dirac spectral function} as 
$$ \zeta_D(s) = \sum_{\lambda >0} \frac{1}{\lambda^{s}}  $$
A {\bf Dirac tree forest ratio} of a graph could now be defined as
$$  \tau_D(G)  = \frac{{\rm Det}(1+D)}{{\rm Det}(D)}  $$
Because non-zero eigenvalues come in pairs $\lambda,-\lambda$, we have
$$ \tau_D(G) = \prod_{\lambda > 0} (1 + \frac{1}{\lambda})(1-\frac{1}{\lambda}) 
             = \prod_{\lambda > 0} (1 - \frac{1}{\lambda^2}) \; . $$
If $1$ is an eigenvalue of $D$ we would exclude that factor. One could also look at the 
connection Laplacian zeta function. 

\paragraph{}
Looking at the zeta function rather than the eigenvalues adds an analytic angle to 
spectral theory. One can look at a {\bf analytic properties} of the function  $\zeta(s)$ 
and especially at the places, where $\zeta(s)$ is singular or places where $\zeta(s)$ 
is zero. The above formula shows that the tree-forest ratio measures a sort of 
growth rate of $\zeta(s)$ along the real axes.

\paragraph{}
Unlike trees or forests, 
the {\bf spectral zeta function} can be considered for manifolds $M$, even so
the interpretation of trees and forests is no more adequate in the continuum. 
On compact Riemannian manifolds $M$ we have an exterior derivative $d$ 
and so a Hodge Laplacian $L=D^2 = (d+d^*)^2$.  Define 
$$ \tau(M) = {\rm Det}(1+L)/{\rm Det}(L) \; , $$
where both ${\rm Det}$ is the Ray-Singer determinant defined by a zeta regularization of 
$L=(d+d^*)^2$. One could also look just at the zeta function defined by the 
scalar Laplacian. We do not see any use yet for the quantity $\tau(M)$, the point
is just that it can be considered. 

\paragraph{}
Let us look for example at the case $M=\mathbb{T}$, where $D=i \frac{d}{dx}$ which 
has the eigenfunctions $e^{i n x}$ with eigenvalues $-n$. The Dirac zeta function
$$ \sum_{n>0} \frac{1}{n^s} \; . $$ 
The Laplacian zeta function is $\sum_{n>0} \frac{1}{n^{2s}}$
as the eigenvalues of $L$ are $1/n^2$. 


\paragraph{}
\begin{propo}
$\tau(\mathbb{T}) = \sinh(\pi)/\pi$.
\end{propo}

\begin{proof}
We have 
$$ \tau(M)  = \prod_{n>0} (1+\frac{1}{n^2}) = \sinh(\pi)/\pi $$
using $\sin(\pi x) = \pi x \prod_{n=1}^{\infty} (1-\frac{x^2}{n^2})$ at $x=i$.
\end{proof} 

\section{The Barycentral limit}

\paragraph{}
The quantity $T(G)=\log(\tau(G))$ is additive with respect to the
disjoint union operation $T(G+H) = T(G)+T(H)$.  Let $|G|$ denote the number
of vertices of $G$. Define the {\bf tree-forest index} 
$$   i(G) = \log(\tau(G))/|G| \;  $$
where $|G|$ is the number of vertices. 
It has a chance to remain finite when taking graph limits. 
Let $G_n$ denote the $n$'th Barycentric refinement of a graph $G=G_0$. 
Now, the growth of the tree forest ratio is dominated by the growth
of the forests. The index $i(G_n)$ is the difference of the 
Forest index $\log(F(G_n))/|G_n|$ and the tree index. 

\begin{thm}[Tree Forest Universality]
$\lim_{n \to \infty} i(G_n)$ exists for every graph $G$ and only depends on
the maximal dimension $d$ of $G$. For $d=1$, the limit is $2\log(\phi)$,
where $\phi$ is the golden ratio. 
\end{thm}

\begin{proof}
We prove this from the Barycental central limit theorem which asserts that there
is a (only dimension-dependent measure) $\mu$ on $[0,\infty)$ such that the 
spectral measures of $L(G_n)$ (which are point measures for all $n$) converge 
weakly to $\mu$. Now, we have $i(G_n) = \int_0^{\infty} \log(1+1/x) d\mu(x)$
which is a difference of two {\bf potential values} 
$$ i(G_n) = \int_0^{\infty} \log(1+x) d\mu(x) - \int_0^{\infty} \log(x) \; d\mu(x) \; .  $$
The $1$-dimensional case is special in that $d\mu$ has compact support on $[0,4]$ and
is the equilibrium measure on that interval. The potential $f(z) = \int \log|a-z| \; d\mu(z)$
exists for every $a \in \mathbb{Z}$ and is subharmonic and zero exactly on the 
interval $[0,\pi]$. The limiting index $i(G)$ in $1$-dimension is therefore the potential 
value at $a=-1$.  \\
In the $d=2$ or higher dimensional case, the measure $\mu$ does no more have compact support. 
This is related to the fact that the maximal vertex degree becomes unbounded under Barycentric 
refinements. While in one dimension, the measure $\mu$ was absolutely continuous with a density
concentrated mostly on the boundary $0,4$ of the support, in higher dimensions, the measure 
has a tail continuity both at the lower bound $0$ as well as at infinity. At zero we need
a mild decay to make sure that $\int_0^{\infty} \log(x) \; d\mu(x)$ is finite. 
At infinity, we need a stronger decay to make sure that $\int_0^{\infty} \log(1+x) d\mu(x)$ exists. 
But we know also from potential theory of subharmonic functions and in our case from the 
combinatorial context that $0 \leq \int_0^{\infty} \log(x) \; d\mu(x) < \int_0^{\infty} \log(1+x) d\mu(x)$. 
It suffices therefore to show that $\int_0^{\infty} \log(1+x) d\mu(x)$ exists. 
This amounts to estimate the decay rate of $\mu$ at infinity.
The fact that there is an exponential decay in dimension $2$ or higher in the next section.
We are using an estimate $\lambda_k \leq 2 d_k$,
a Perron-Frobenius result on the limiting dimension distribution in Barycentric limits
as well as an estimation of the vertex cardinality points of $G_n$ which depends on the
{\bf generation} of a point. A vertex in $G_n$ has generation $k$ if it has become a vertex
in $G_{n-k}$ and not before. A vertex in $G_n$ of {\bf generation} zero is a vertex which has
in the previous level $G_{n-1}$ not been a vertex. Note that if a simplex is a vertex, it is
especially again a vertex in the next generation. What happens is that we have an 
exponential decay of vertex degrees depending on generation. 
\end{proof}

\paragraph{}
A different proof could be done along the lines of the Barycentral spectral limit
theorem without using the spectral measure $d\mu$:
if $G$ has maximal dimension $d$ then the Barycentric refinement
$G_1$ of a $d$-simplex in $G$ consists of $(d+1)!$ simplices glued.
There is a stratification of growth rates of the different simplices
depending on dimension. The number of maximal simplices grows exponentially
faster than the number of co-dimension-one simplices, we can asymptotically treat every 
refined $n$-simplex as a copy of disjoint $(n+1)!$ simplices plus a gluing correction
which as it is lower dimensional will be a definite factor smaller. The limiting tree forest
ratio is again Cauchy sequence as we add up a sequence of errors which geometrically decay. 

\section{Tail decay}

\paragraph{}
In this section we prove a proposition on the 
limiting density of states $dk$ in dimension $d \geq 2$. 

\begin{propo}
There exists a positive constant $c$ such that 
$\mu([x,\infty)) \leq e^{-c x}$. 
\end{propo}

\paragraph{}
In dimension $d=1$, the density of states $dk$ has compact 
support $[0,4]$ and therefore does not need an exponential 
decay estimate. The lemma of course still applies.
Let therefore $G$ be a graph of dimension $d \geq 2$
and let $G_n$ be the Barycentric refinements of $G$. The vertices 
of $G_n$ are the simplices (complete subgraphs) of  $G_{n-1}$ and two 
are connected, if one is contained in the other.  Every vertex $x$ in $G_n$
can be assigned a {\bf generation} $\gamma(x)$ defined to be the integer $k$,
if it has been a vertex in $G_{n-1}, \dots, G_{n-k+1}$.  
For example, a vertex has generation $1$ if it has been a simplex of
positive dimension in $G_{n-1}$.  A point has generation $n$ if has
been a vertex in $G$ already. Every vertex in $G_n$ has generation 
$\gamma(x) \in \{1,2,\dots, n\}$. 

\paragraph{}
The $f$-vector of $G$ is the vector $(f_0,f_1, \dots, f_d)$, where
$f_k$ are the number of $k$-dimensional simplices in $G$. 
There are ${\rm gen}_n = f_0(G_0)$ vertices of generation $n$. \\
There are ${\rm gen}_{n-1}=f_0(G_1)-{\rm gen}_n$ vertices of generation $n-1$.  \\
There are ${\rm gen}_{n-2}=f_0(G_2)-{\rm gen}_{n}-{\rm gen}_{n-1} = f_0(G_2)-f_0(G_1)$ vertices of generation $n-2$  \\
There are ${\rm gen}_1    =f_0(G_n)-{\rm gen} {n}-{\rm gen}_{n-2} ..-gen_2$ vertices of generation $1$. \\
Since $f_0(G_n)$ grows exponentially at least like $(d+1)!^n$, we also have in $G_n$
at least $C (d+1)!^{n-k}$ vertices of generaton $k$. 

The {\bf Barycentric refinement operator} $A$ is a $(d+1) \times (d+1)$ upper
triangular matrix defined by $f(G_1)=A f(G)$. 
It is given as $A_{ij} = {\rm Stirling}(i,j) i!$, involving the {\bf Stirling numbers 
of the second kind}. The {\bf Perron-Frobenius eigenvector} of $A$ is the 
probability eigenvector of $A$ belonging to the largest eigenvalue $(d+1)!$. 
For example, in the case $d=2$, when 
$A=\left[ \begin{array}{ccc} 1 & 1 & 1 \\ 0 & 2 & 6 \\ 0 & 0 & 6 \\ \end{array} \right]$,
the Perron-Frobenius eigenvector is $[1/6,1/2,1/3]$ to the eigenvalue $(d+1)!=6$. 
This means that for large $n$, the graph $G_n$ has about $1/6$th of the simplices
which are points, about a half which are edges and about $1/3$rd which are triangles. 
In the limit $d \to \infty$, the Perron Frobenius probability measures appear converge weakly 
to a point measure around $0.7215$, a fact which we could not prove yet. 

\begin{lemma}
There is a $c_1>0$ such that for all k, at least $c_1^{n-k} |G_n|$ vertices
in $G_n$ are generation $k$ vertices. 
\end{lemma}

This means that the generation $\gamma$ as a random variable on the finite set 
of vertices of $G_n$ asymptotically has an exponential distribution. There
very few vertices of large generation and very many of small generation. 
Now we will see that on large generation vertices, there is an upper bound on 
the eigenvalues which is exponential too in the generation. 

\paragraph{}
If $S(x)$ is the {\bf unit sphere} of a vertex $x$, then
the k'th Barycentric refinement $S_k(x)$ is the unit sphere of $x$ in $G_k$.
Because $|S_k(x)|$ is the {\bf vertex degree} of $x$, the vertex degree of 
generation $k$ vertices is at least $|S_k(x)| \sim d!^k$. 

\begin{lemma}
There exists $c_2>0$ such that maximally $c_1^k |G_n|$ vertices have
degree larger than $c_2^k$.
\end{lemma}

\paragraph{}
It follows from the Gershrogin circle theorem 
that there less than $c_1^k |G_n|$ vertices for which
the largest eigenvalue is larger than $2 c_2^k$. Now $c_1=1/(d+1)!$ and
$c_2=d!$. Since this implies that the error area decays exponentially 
for $x \to 1$, was also have the function itself decay exponentially. 

\section{Questions}

\paragraph{}
Which graphs with $n$ vertices have the maximal tree-forest ratio? 
The minimum is $1$. For small $n$ we see that the linear graph is the
maximum. It is pretty safe to conjecture that the linear graph is
always the maximum:

\paragraph{}
What is the average tree-forest ratio on Erdoes-Renyi spaces? 
We can also look more generally at the potential average on 
Erdoes-Renyi spaces which produces a density of states measure
for each $n$ and $p$. And then we can ask for which choice of $p_n$
does the density of states average $dk_{n,p(n)}$ converge weakly 
in the limit $n \to \infty$.

\paragraph{}
What happens with graph operations like various products or joins? 
For disjoint unions, we have the product of determinants so that 
$\tau(G+H)=\tau(G) \tau(H)$. 
Under graph complements, there is no clear relation yet. Also
the Zykov join operation does not produce obvious relations.
What happens with the tree-forest ratio when taking Shannon products?
We do not see any (at least obvious) relation between the tree number
yet. 

\paragraph{}
For connection Laplacians, we know that the zeta functions multiply. 
because the connection Laplacian of a product of graphs tensor. 
$L_{G * H} = L_G \otimes L_H$. 
We therefore have for the connection Laplacians:
$$ \log(\rho(G * H)) = \sum_k (-1)^{k+1} \zeta_G(k) \zeta_H(k)/k \; . $$

\paragraph{}
What is the relation between the Hodge Laplacian and connection Laplacian 
tree-forest ratio? In order to define a tree-forest ratio from the connection
Laplacian itself, we have to look at the spectrum of the square.  
We can for any positive energized complex look at 
$\tau_L(G) = \det(L+1)/\det(L)$ which makes sense as $L$ is always invertible. 

\paragraph{}
We should reiterate that in the case $d \geq 2$, we have not even a proof
yeat that there are gaps in the spectrum. This should be easy to do as 
in the case $d=2$ we see a prominent gap $[4,6]$ associated to the integrated
density of states $1/2$. While we know now something about the decay of $dk$ at
infinity, we do not have an asymptotic near $0$ or the end of gaps. As for $0$,
the Schur estimate does not help as as we have always a certain percentage of 
vertices with minimal vertex degree $(d+1)!$. That means that we can (using Shur)
only prove that there exists a constant $c$ such that $dk([0,x]) \leq c x$ for 
small $x$. As we know $U(0) \geq 0$ due to its relation to the growth rate of trees,
we know that $U(z)$ is finite in a neighborhood of $z=0$. It is natural to conjecture
that the potential $U(z)< \infty$ for all $z \in \mathbb{Z}$. This would imply that the 
{\bf logarithmic capacity} $-\int \int \log|z-w| dk(z) dk(w)$ is finite in all 
dimensions $d$. We know that the value is $0$ in the case $d=0$, where
$U(z)=0$ on the support $[0,4]$ of $dk$. Having finite logarithmic capacity especially 
implies that the {\bf Lebesgue decomposition} of $dk=dk_{pp} + dk_{sc} + dk_{ac}$ 
has no pure point part. 

\paragraph{}
Fourier theory could be an approach to determine the nature of the spectrum $d\mu$.
Since the measure $\mu$ has exponential decay at infinity, all Fourier coefficients
$\hat{d\mu}(t) =\int_0^{\infty} e^{i t x} \; d\mu(x)$ exist. In order to determine
the spectral type, one could also focus on a compact part of the spectrum like 
for example, $[0,4]$. Now, we can look at Fourier series of $d\mu$ instead rather
than the Fourier transform. 
For example, by a theorem of Wiener \cite{Katznelson}, if 
$\frac{1}{n} \sum{k=0}^{n-1} \hat{\mu}(k)^2$ converges to zero for $n \to \infty$, 
then $d\mu$ has no point spectrum. 

\section{Code}

\paragraph{}
The following Mathematica code plots the potential function
(the analog of the Lyapunov exponent), and the integrated
density of states (the analog of the rotation number) as well as the spectrum of 
the Barycentric measure $dk$ in the planar case. The code is optimized
for the two-dimensional case so that we do not have to use general costly 
clique finding algorithms. 

\begin{tiny} \lstset{language=Mathematica} \lstset{frameround=fttt}
\begin{lstlisting}[frame=single]
T3[s_]:=Module[{v=VertexList[s]},Union[Partition[Flatten[Union[
  Table[z=EdgeList[Subgraph[s,Complement[VertexList[
  NeighborhoodGraph[s,v[[k]]]],{v[[k]]}]]]; Table[Sort[
  {v[[k]],z[[j,1]],z[[j,2]]}],{j,Length[z]}],{k,Length[v]}]]],3]]];
T2[s_]:=Module[{e=EdgeList[s]},
  Table[Sort[{e[[k,1]],e[[k,2]]}],{k,Length[e]}]];
T1[s_]:=Module[{v=VertexList[s]},Table[{v[[k]]},{k,Length[v]}]]; 
Whitney2D[s_]:=Sort[Map[Sort,Union[T3[s],T2[s],T1[s]]]];
GraphFromComplex[G_]:=Module[{n=Length[G],v,e={}},Do[x=Sort[G[[k]]]; 
  If[Length[x]==3,e=Union[e,{x->{x[[1]]},x->{x[[2]]},x->{x[[3]]},
  x->Sort[{x[[1]],x[[2]]}],x->Sort[{x[[1]],x[[3]]}],
  x->Sort[{x[[2]],x[[3]]}]}]];
  If[Length[x]==2,e=Union[e,{x->{x[[1]]},x->{x[[2]]}}]],{k,Length[G]}];
  rules=Table[G[[k]]->k,{k,n}]; UndirectedGraph[Graph[e /. rules]]];
Barycentric2D[s_]:=GraphFromComplex[Whitney2D[s]];
Barycentric2D[s_,k_]:=Module[{s1=s},Do[s1=Barycentric2D[s1],{k}];s1];
s=Barycentric2D[CompleteGraph[3],6];
L=Sort[Eigenvalues[1.0*KirchhoffMatrix[s]]];
U[z_]:=Sum[Log[Abs[z-L[[k]]]],{k,Length[L]}]/Length[L]; 
V[x_]:=Sum[If[L[[k]]<x,1,0],{k,Length[L]}]/Length[L];
f[x_]:={Red,Point[{x,0}]}; 
W = Graphics[Map[f,L], PlotRange->{{-1,10},{-1,7}}];
Show[{W,Plot[{3*U[x],6*V[x]},{x,-2,10},PlotPoints ->1000]}]
{TreeIndex=U[0],ForestIndex=U[-1],TreeForestIndex=TreeIndex/ForestIndex}
\end{lstlisting}
\end{tiny}

\paragraph{}

\begin{figure}[!htpb]
\scalebox{1.2}{\includegraphics{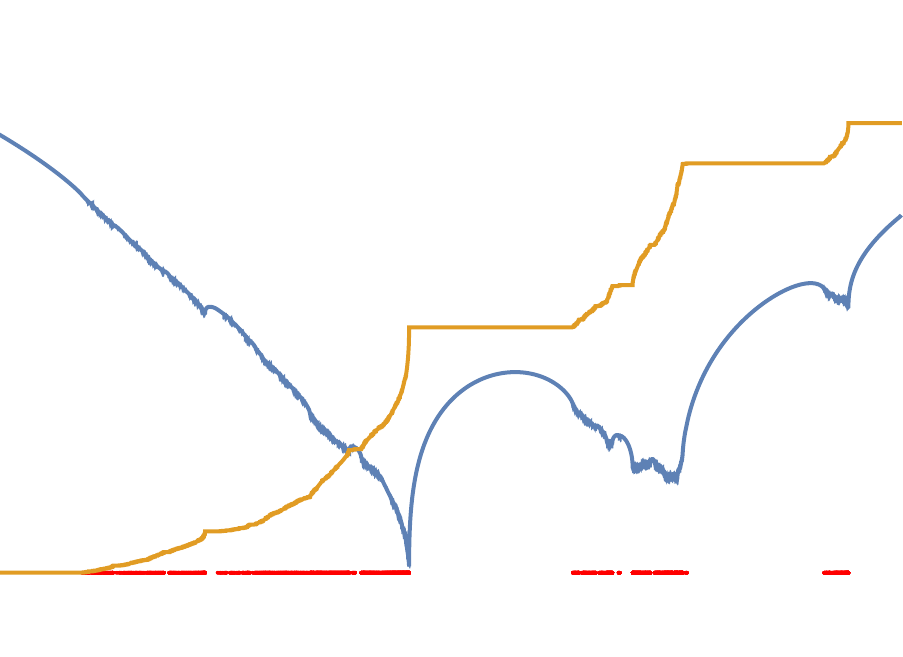}}
\label{potential}
\caption{
The output of the above code shows the potential ${\rm Re}(U(x+0i))$
and the integrated density of states ${\rm Im}(U(x+0i)$ of the 
complex potential $U(z)=\int_{\mathbb{C}} \log(z-w) \; dk(w)$
of the Barycentric 2-dimensional density of states $dk$. We 
see the 6th Barycentric refinement $G_6$ for $G=G_0=K_3$. 
The prominent gap $[4,6]$ is very clear numerically, 
but not theoretically established. 
}
\end{figure}

\begin{figure}[!htpb]
\scalebox{1.0}{\includegraphics{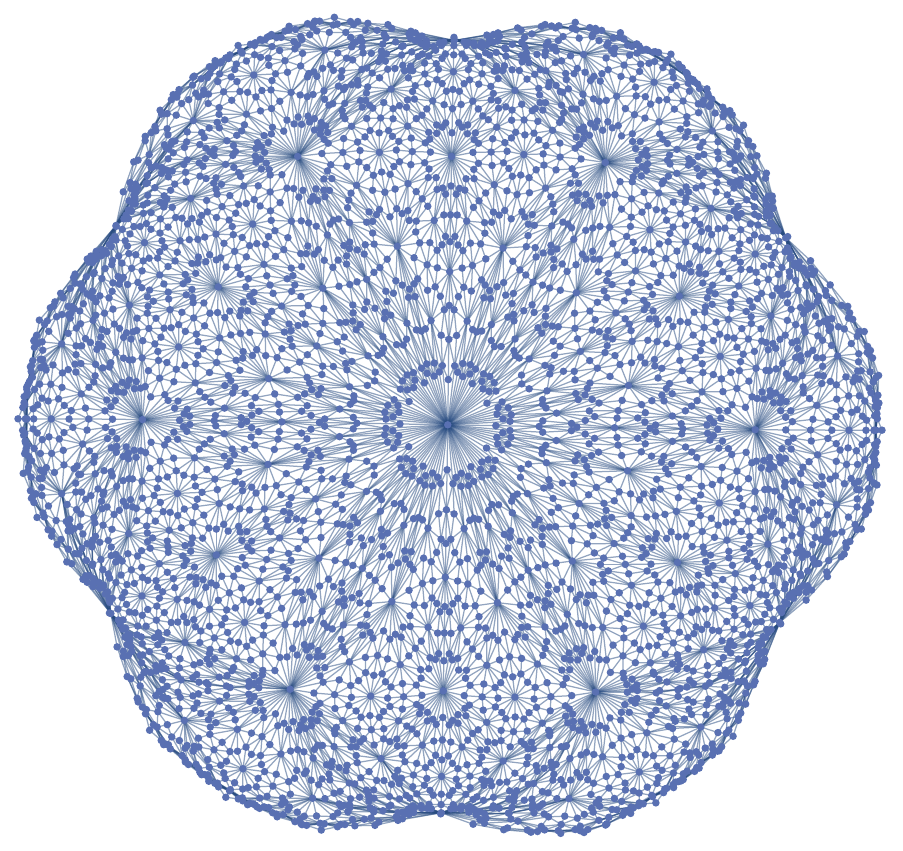}}
\label{potential}
\caption{
The 5'th Barycentric refinement $G_5$ of $G_0=K_3$ has 3937 vertices, 
11712 edges and 7776 triangles. The next version $G_6$ used in the above spectral 
picture already has $23425=3937+11712+7776$ vertices. 
}
\end{figure}

\bibliographystyle{plain}

\end{document}